\documentclass[useAMS]{gDEA2e}
\newcommand{\ba}{\begin{array} }
\newcommand{\ea}{\end{array} }
\newcommand{\bae}{\begin{eqnarray}}
\newcommand{\eae}{\end{eqnarray}}
\newcommand{\bea}{\begin{eqnarray*}}
\newcommand{\eea}{\end{eqnarray*}}
\newcommand{\be}{\begin{equation}}
\newcommand{\ee}{\end{equation}}

\newcommand{\modifyb}[1]{\textcolor{black}{#1}}
\newcommand{\modifyr}[1]{\textcolor{black}{#1}}

\usepackage{graphicx}
\usepackage{color}
\usepackage[colorlinks]{hyperref}
\usepackage{lscape}
\usepackage{graphicx}
\usepackage{epstopdf}
\begin{document}
\doi{10.1080/1023619YYxxxxxxxx}
 \issn{1563-5120}
\issnp{1023-6198} \jvol{00} \jnum{00} \jyear{2011} \jmonth{October}

\markboth{Taylor \& Francis and I.T. Consultant}{Journal of Biological Dynamics}

 \title{\modifyr{Global Dynamics of a Discrete Two-species Lottery-Ricker Competition Model}}
\author{Yun Kang\thanks{$^\ast$Corresponding author. Email: Yun.Kang@asu.edu}\\\vspace{3pt}  {\em{Applied Sciences and Mathematics, Arizona State University, \\Mesa, AZ 85212, USA. E-mail: yun.kang@asu.edu}}\\\vspace{6pt} Hal Smith\\\vspace{3pt}{\em{School of Mathematical and Statistical Sciences, Arizona State University,\\ P.O. Box 871804, Tempe, AZ 85287, USA.  E-mail: halsmith@asu.edu}}}
 \maketitle
 \begin{abstract}
In this article, we study the global dynamics of a discrete two dimensional competition model. We give sufficient conditions on the persistence of one species and the existence of local asymptotically stable interior period-2 orbit for this system. Moreover, we show that {for a} certain parameter range, there exists a compact interior attractor that attracts all interior points except a Lebesgue measure zero set. This result gives a weaker form of coexistence which is referred to as \emph{relative permanence}. This new concept of coexistence combined with numerical simulations strongly suggests that the {basin} of attraction of the locally asymptotically stable interior period-2 orbit is {an} infinite union of connected components. This idea may apply to many other ecological models. Finally, we discuss the generic dynamical structure that gives \emph{relative permanence}. \bigskip
\begin{classcode}Primary 37B25; 39A11; 54H20; Secondary 92D25\end{classcode}\bigskip
\end{abstract}
\section{A discrete two species competition model}\label{sec:model}
Mathematical models can provide important {insight} into the general conditions that permit the coexistence of competing species and the situations that lead to competitive exclusion (Elaydi and Yakubu 2002). A  model of resource-mediated competition between two competing species can be described as follows (Adler 1990; Franke and Yakubu 1991a, 1991b)
\bae\label{gx}
x_{n+1}&=&\frac{r_1 x_n}{a+x_n+y_n}\\
\label{gy}
y_{n+1}&=&y_n e^{r_2-(x_n+y_n)}
\eae where $x_n$ and $y_n$ denote the population sizes of two competing species $x$ and $y$ at generation $n$ respectively; all parameters $r_1, r_2, a$ are strictly positive. Franke and Yakubu (1991a) established the ecological principle of mutual exclusion as a mathematical theorem in a general discrete two-species competition system including \eqref{gx}-\eqref{gy}. They (1991b) also gave an example that such exclusion principle fails where two species can coexist through a locally stable period-2 orbit. This phenomenon of coexistence has been observed in many other competition models (e.g., Yakubu 1995, 1998; Elaydi and Yakubu 2002a, 2002b) including system \eqref{gx}-\eqref{gy} with $a=0$:
\bae\label{gx0}
x_{n+1}&=&\frac{r_1 x_n}{x_n+y_n}\\
\label{gy0}
y_{n+1}&=&y_n e^{r_2-(x_n+y_n)}
\eae
Notice that the equation \eqref{gx0} is the non-overlapping \modifyb{lottery model (Chesson 1981)} with singularity at the origin. Every initial condition $(x_0,0)$ with $x_0>0$  maps to $(r_1,0)$. \modifyb{The lottery model emphasizes the role of chance. It assumes that resources are captured at random by recruits from a larger pool of potential colonists (Sale 1978; chapter 18, Chain \emph{et al} 2011).} When $y_n=0$, \eqref{gx0} can be a \modifyb{reasonably} good approximation for plant species where a single individual can sometimes grow very big in the absence of competition from others or \modifyb{ for a territorial marine species, such as coral reef fish, where a single individual puts out huge number of larvae} (communications with P. Chesson; also see Sale 1978). Chesson and Warner (1981) used such non-overlapping lottery models to study competition of species in a temporally varying environment. In this paper, we focus on the dynamics of \eqref{gx0}-\eqref{gy0}. \modifyr{The system \eqref{gx0}-\eqref{gy0} may be an appropriate model for resource competition between a territorial species $x$ and a non-territorial species $y$}.

A recent study by Kang (submitted to JDEA) shows that \eqref{gx0}-\eqref{gy0} is persistent with respect to the total population of two species, i.e., all initial conditions in $\mathbb R^2_+\setminus\{(0,0)\}$ are attracted to a compact set which is bounded away from the origin.  The results obtained in Kang (preprint 2010) allow us to explore the structure of the {basin} of attraction of the asymptotically stable period-2 orbit of the system \eqref{gx0}-\eqref{gy0} lying in the interior of the quadrant. In this article,  we study the global dynamics of \eqref{gx0}-\eqref{gy0}. The objectives of our study are two-fold:
\begin{enumerate}
\item Mathematically, it is interesting to study the global dynamics of \eqref{gx0}-\eqref{gy0} since it has singularity at the origin. Thus, the first objective of our study is to give sufficient conditions for competitive exclusion and coexistence of  \eqref{gx0}-\eqref{gy0}.
\item Biologically, it is very important to classify and give sufficient \modifyb{conditions for} the coexistence of species in ecological models. Among many forms of coexistence, permanence is the strongest concept since it requires all strictly positive initial conditions converge to the bounded interior attractor. Although permanence fails for \eqref{gx0}-\eqref{gy0}, we establish the weaker notion \emph{relative permanence}: almost all (relative to Lebesgue measure) strictly positive initial conditions converge to the bounded interior attractor. Numerical simulations of other ecological models (e.g., Franke and Yakubu 1991; Kon 2006;  Cushing, Henson and Blackburn 2007; Kuang and Chesson 2008) suggest the possibility that relative permanence may apply where permanence fails. Our second objective of this article is to draw attentions on the concept of \emph{relative permanence}. Our study could potentially provide insight on weaker forms of coexistence for general ecological models and open problems on the basins of attractions of stable cycles for a discrete competition model studied by Elaydi and Yakubu (2002a).
 \end{enumerate}

Simple analysis combined with numerical simulations suggest the following interesting dynamics of the system \eqref{gx0}-\eqref{gy0}
\begin{enumerate}
\item There is no interior fixed point. \modifyb{The eigenvalue governing the local transverse stability of the boundary equilibrium on the $x$-axes (i.e., $y=0$) is given by $e^{r_2-r_1}$ (or $\frac{r_1}{r_2}$ on the $y$-axes). If this eigenvalue is less than 1, then we say that the equilibrium point on the $x$-axes (or $y$-axes) is transversally stable, otherwise, it is transversally unstable.} Thus, if $r_1> r_2$, then the boundary equilibrium $\xi^*=(r_1,0)$ is transversally stable and $\eta^*=(0,r_2)$ is transversally unstable; while $r_1<r_2$, $\xi^*=(r_1,0)$ is transversally unstable and $\eta^*=(0,r_2)$ is transversally stable.
\item For a certain range of $r_1$ and $r_2$ values, there exists an asymptotically stable periodic-2 orbit in the interior of the quadrant which attracts almost every interior point in $\mathbb R^2_+$. For example, when $r_1=2,r_2=2.2$, the periodic-2 orbit is given by
$$(x_1^i,y_1^i)=(0.1536, 2.9629) \mbox{ and } (x_2^i, y_2^i)=(0.0986, 1.1849)$$ and \modifyb{the eigenvalues of the product of the Jacobian matrices along the orbit} are $0.91$ and  $0.26$.
\item There {exists} a heteroclinic orbit connecting $\xi^*$ to $\eta^*$ (see Figure \ref{hc_orbitd});
\begin{figure}[ht]
\centering
\includegraphics[width=100mm]{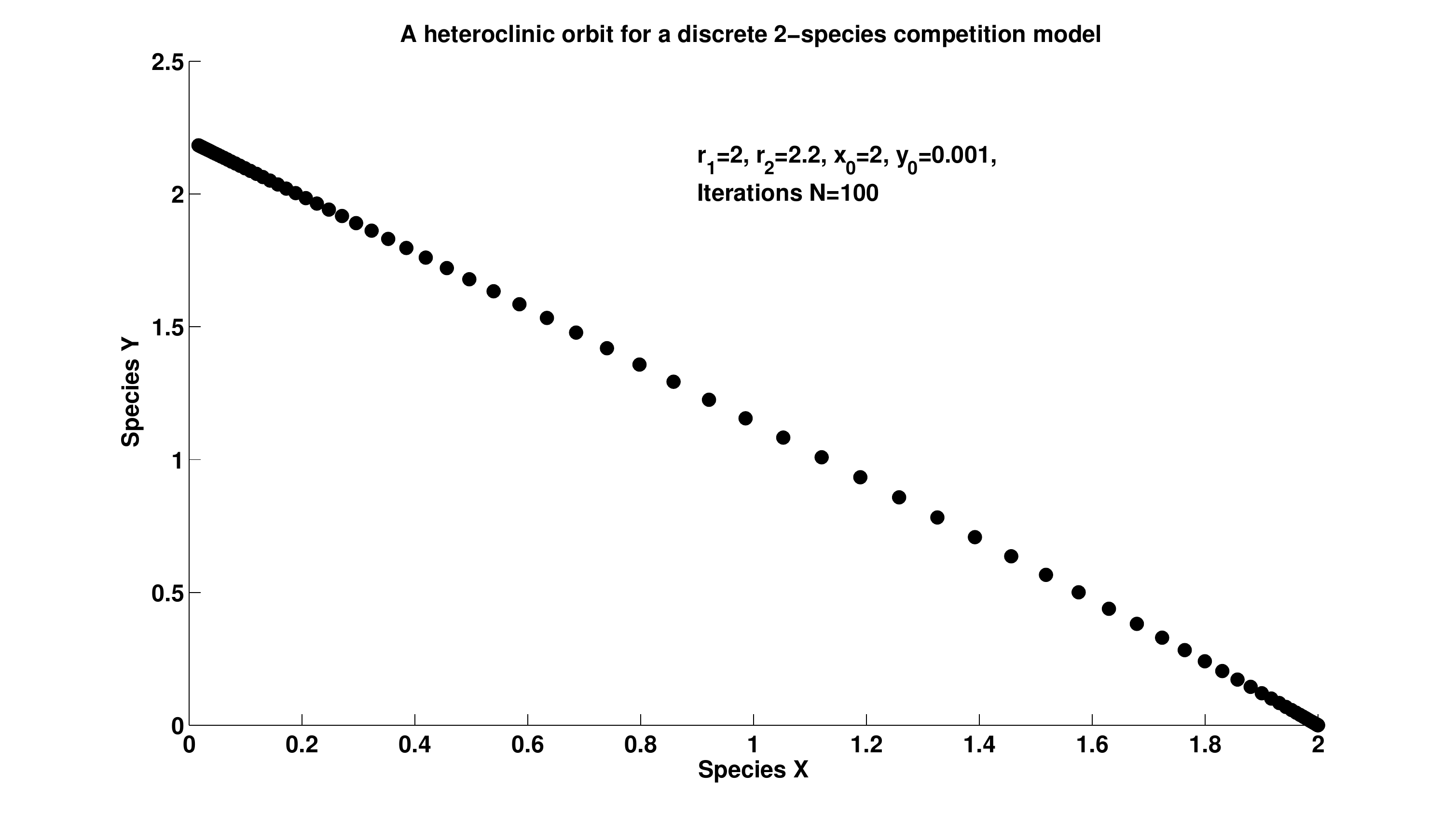}
 \caption{A heteroclinic orbit of the system \eqref{gx0}-\eqref{gy0} when $r_1=2, r_2=2.2, x_0=2, y_0=0.001$.}
 \label{hc_orbitd}
\end{figure}

\item The {basin} of attraction of the interior periodic-2 orbit $P_2^i$ consists of all interior points of $\mathbb R^2_+$ except all the pre-images of the heteroclinic curve $C$ (where $C$ is the closure of the union of all heteroclinic orbits, see Figure \ref{basind}).
\begin{figure}[ht]
\centering
\includegraphics[width=100mm]{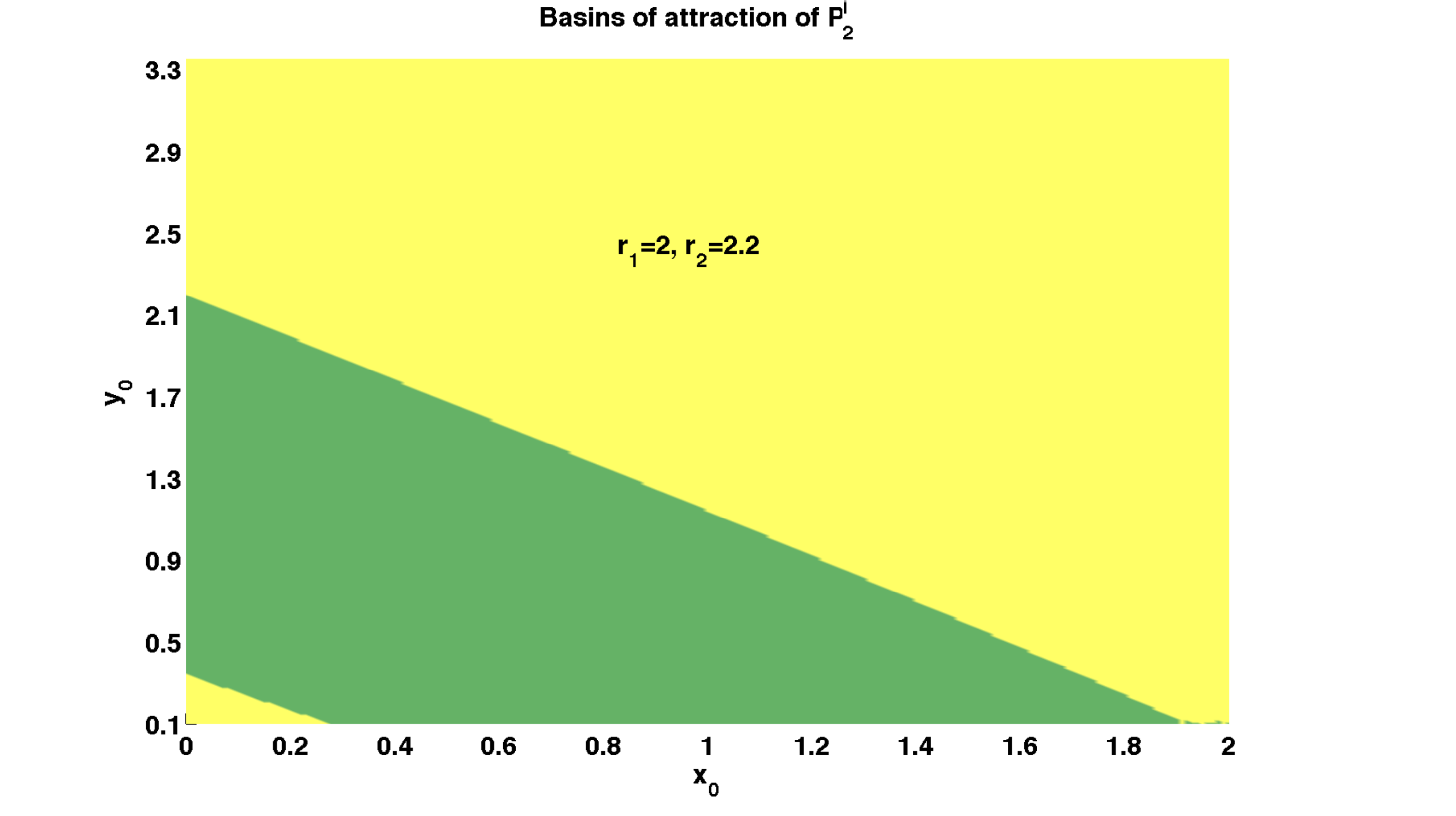}
 \caption{The basin of attraction of the interior period-2 orbit is the open quadrant minus the pre-images of the heteroclinic curve $C$. The latter partition the quadrant into components which are colored according to which of the two periodic points attract points in the component under the second iterate of the map. Given a point in one of the regions, there is a large number $N$, such that the point  will be very close to $(x_1^i,y_1^i)$ at the iteration $t$ and will be very close to $(x_2^i,y_2^i)$ at the iteration $t+1$ for all $t>N$.}
 \label{basind}
\end{figure}
\end{enumerate}
Moreover,  further analysis and numerical simulations suggest that if the system \eqref{gx0}-\eqref{gy0} satisfies the following conditions \textbf{C1-C3}, then it has the same global dynamics as the case $r_1=2, r_2=2.2$:
\begin{itemize}
\item\textbf{C1:} The values of $r_1, r_2$ satisfy $$2<r_2<2.52,\, r_2>r_1>1, \mbox{ and }\,e^{2r_2-1-e^{r_2-1}}>1.$$
\item\textbf{C2:} There is a boundary period-2 orbit $M_y=\{\eta_1,\eta_2\}=\{(0,y_1),(0,y_2)\}$ where $\frac{r_1^2}{y_1y_2}>1$.
\item\textbf{C3:} There is a heteroclinic orbit connecting $\xi^*$ to $\eta^*$ (see Figure \ref{hc_orbitd}).  
\end{itemize}
\noindent Condition \textbf{C1} implies that the equilibria $\xi^*$ and $\eta^*$ of the system \eqref{gx0}-\eqref{gy0} are saddle \modifyb{nodes}, where $\xi^*$ is transversally unstable and $\eta^*$ is transversally stable. Moreover, species $y$ can invade species $x$. Condition $2<r_2<2.52$ combined with Condition \textbf{C2} indicates that species $x$ can invade species $y$ on its periodic-2 orbit $\{(0,y_1),(0,y_2)\}$. Figure \ref{rp2} describes the schematic scheme of the global dynamics of the system \eqref{gx0}-\eqref{gy0} when it satisfies Condition \textbf{C1}-\textbf{C3}.
\begin{figure}[ht]
\begin{center}
   \includegraphics[width=100mm]{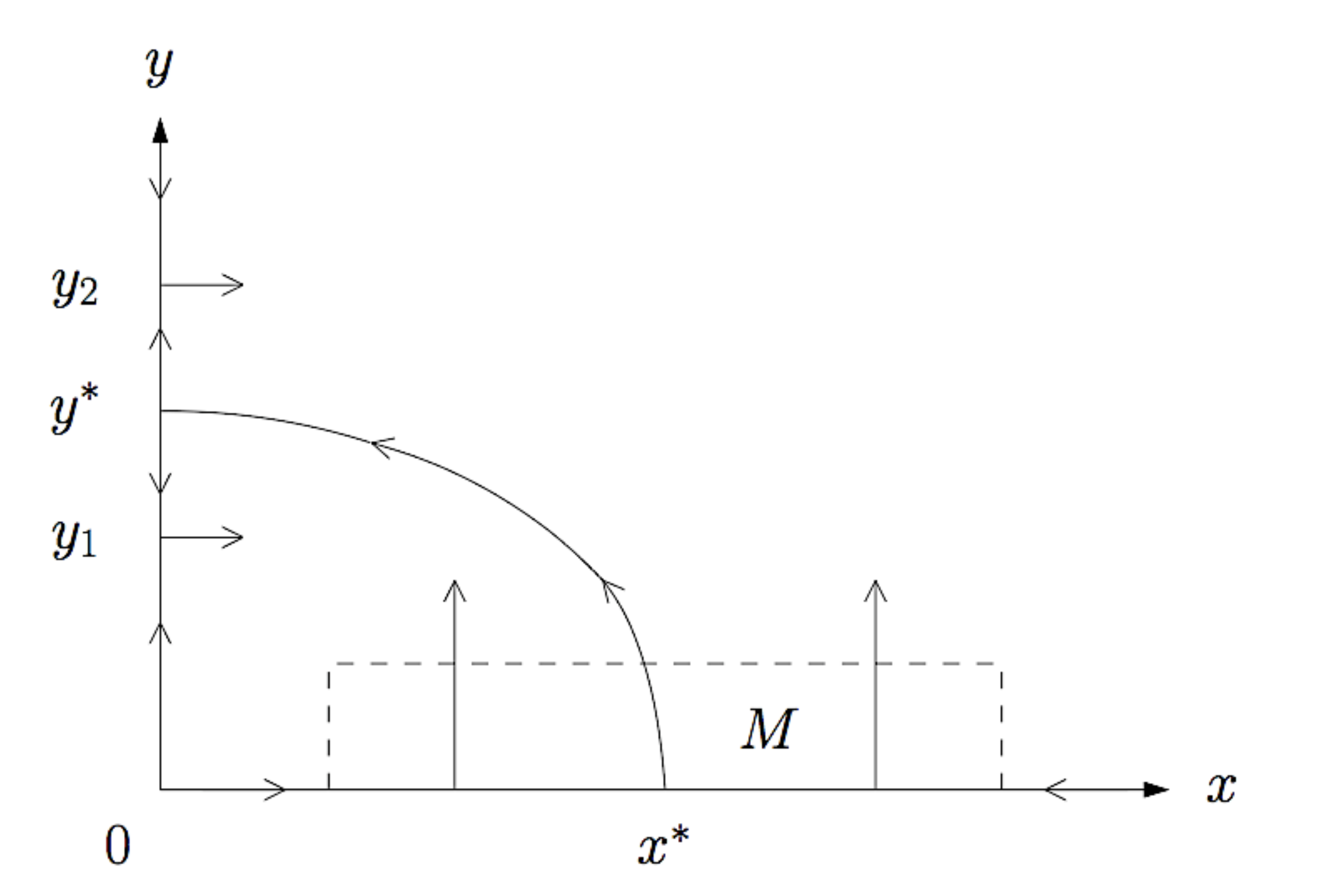}
\caption{Schematic features of the system \eqref{gx0}-\eqref{gy0} when $r_1=2.,r_2=2.2$.}
 \label{rp2}
 \end{center}
   \end{figure}

\indent The structure of the rest of the article is as follows: In Section \ref{sec:ps}, 
we give the basic notations and preliminary results that will be used in proving our main {theorems. In} 
Section \ref{sec:ce}, we obtain sufficient conditions on the persistence of one species and the extinction of the other species 
by using Lyapunov functions (Theorem \ref{th1:ce}){. In} Section \ref{sec:rp} we first give a sufficient condition on the existence 
of locally asymptotically stable interior period-2 orbit for the system \eqref{gx0}-\eqref{gy0} (Theorem \ref{th2:bifurcation}); 
then we show that {for a} certain parameter range, the system \eqref{gx0}-\eqref{gy0} is \emph{relative permanent}, i.e., 
it has a compact interior attractor that attracts almost points in $\mathbb R^2_+$ (Theorem \ref{th3:rp}) by applying theorems 
from persistent theory. In the last section \ref{sec:con} we discuss {the fact that the} global dynamics of the system \eqref{gx0}-\eqref{gy0} 
are generic rather than rare.  {Similar dynamic behaviors of \eqref{gx0}-\eqref{gy0} have been observed in many biological models}. 
Studying sufficient conditions for the relative permanence of the generalization of such biological models can be our future direction.

\section{Notations and preliminarily results}\label{sec:ps}
Notice that the system \eqref{gx0}-\eqref{gy0} has singularity at the origin $(0,0)$, thus its state space is defined as $X=\{(x,y)\in \mathbb R^2_+:\,0<x+y<\infty\}$. {Let $H$ denote the map defined by \eqref{gx0}-\eqref{gy0}.} Then $H:\,X\rightarrow X$ is a discrete semi-dynamical system where $H^0(\xi_0)=\xi_0=(x_0,y_0)$ and $H^n(\xi_0)=\xi_n=(x_n,y_n), n\in \mathbb Z_+$. Here, we give some definitions that {will be used} in the rest of the article.
\begin{definition}\label{de:pre-pt}[\emph{Pre-images of a Point}]
For a given point $\xi_0\in X$, we say $\xi\in X$ is a rank-$k$ pre-image of $\xi_0$ if $H^k(\xi)=\xi_0$. The collection of rank-$k$ ($k\geq 1$) pre-images of $\xi_0$ is defined as $$H^{-k}(\xi_0)=\{\xi\in X: H^k(\xi)=\xi_0\}$$
and the collection of all pre-images of $\xi_0$ (including $k=0$) is defined as $$EF_{\xi_0}=\left(\bigcup_{k\geq 1} H^{-k}(\xi_0)\right)\bigcup \{\xi_0\}.$$
\end{definition}
{\begin{definition}\label{de:invariant-set}[\emph{Invariant Set}] We say $M\subset X$ is an invariant set of $H$ if $H(M)=M$.
\end{definition}}
\begin{definition}\label{de:pre-set}[\emph{Pre-images of an Invariant Set}]
Let $M$ be an invariant set for the system \eqref{gx0}-\eqref{gy0}, then $H^{0}(M)=H(M)=M$. The collection of rank-$k$ pre-images of $M$ ($k\geq 1$) is defined as
$${H^{-k}(M)=\bigcup_{\xi_0\in M}\{\xi\in X\setminus M: H^k(\xi)=\xi_0\}};$$ and the collection of all pre-images of $M$ (including $k=0$) is defined as

$${EF_M=\left(\cup_{k\geq 1}H^{-k}(M)\right) \bigcup M=\left[\bigcup_{k\geq 1}\left(\bigcup_{\xi_0\in M}\{\xi\in X\setminus M: H^k(\xi)=\xi_0\}\right)\right]\bigcup M.}$$

\end{definition}
\noindent \textbf{Note:} If $M$ is an invariant set of $H$, then $H^{-k}(M)$ should not contain points in $M$ for all $k\geq 1$.

\begin{definition}\label{de:uniform-repeller}[\emph{Uniform Weak Repeller}] Let $\tilde{X}$ be a positively invariant subset of $X$. We call the compact invariant set $C$ is a \emph{uniformly weak repeller} with respect to $\tilde{X}$ if  there exists some $\epsilon>0$ such that
$$\limsup_{n\rightarrow\infty} d\left( H^n(\xi), C\right)>\epsilon \mbox{ for any } \xi\in \modifyb{\tilde{X}\setminus C}. $$\end{definition}
\begin{definition}\label{de:weak-persistent}[\emph{Uniform Weak $\rho$-Persistence}] Let $\tilde{X}$ be a positively invariant subset of $X$. The semi-flow $H$ is called \emph{uniformly weakly $\rho$-persistent} in $\tilde{X}$ if there exists some $\epsilon>0$ such that
$$\limsup _{n\rightarrow\infty} \rho\left(H^n(\xi)\right)>\epsilon \mbox{ for any } \xi\in \tilde{X}$$
 where $\rho: X\rightarrow \mathbb R_+$ is a \modifyr{persistence function (e.g., $\rho(x,y)=x$ can be a persistence function if we want to study whether species $x$ is \emph{uniformly weakly persistent} or not)}.  We say species $x$ is \emph{uniformly weakly persistent} in $\tilde{X}$ if there exists a $\epsilon>0$ such that $$\limsup _{n\rightarrow\infty} x_n>\epsilon \mbox{ for any } \xi\in \tilde{X}.$$\end{definition}
\begin{definition}\label{de:persistent}[\emph{Uniform Persistence}] Let $\tilde{X}$ be a positively invariant subset of $X$. We say species $x$ is \emph{uniformly persistent} in $\tilde{X}$ if there exists some $\epsilon>0$ such that
$$\liminf _{n\rightarrow\infty} x_n >\epsilon \mbox{ for any } \xi\in \tilde{X}.$$
\end{definition}
\begin{definition}\label{de:permanence}[\emph{Permanence}] Let $\tilde{X}$ be a positively invariant subset of $X$. We say the system $H$ is \emph{permanent} in $\tilde{X}$ if there exists some $\epsilon>0$ such that
$$\liminf _{n\rightarrow\infty}\min\{ x_n,y_n\} >\epsilon \mbox{ for any } \xi\in \tilde{X}.$$
\end{definition}

\begin{definition}\label{de:relative-p}[\emph{Relative Permanence}] We say the system $H$ is \emph{relative permanent} in $X$ if there exists some $\epsilon>0$ such that $\liminf_{n\rightarrow\infty} \min\{x_n,y_n\}>\epsilon$ for almost all initial condition taken in $X$ (i.e., all initial conditions in $X$ except a Lebesgue measure zero set).
\end{definition}

\begin{lemma}\label{l1:pi}[Compact Positively Invariant Set]
Assume that $r_1\neq r_2$, then for any $$0<\epsilon\leq\min\{r_1, r_2, e^{2r_2-1-e^{r_2-1}},r_1e^{r_2-r_1}\}=r^m,$$the compact region defined by
$$D_\epsilon=\left\{(x,y)\in X: \epsilon\leq\,x+y\leq \max\{r_1, e^{r_2-1}\}\right\}$$is positively invariant and attracts all points in $X$.\end{lemma}

\begin{lemma}\label{l2:in_pi}[Pre-images of Invariant Smooth Curve]
Assume that $$r_2>r_1>1\,\mbox{ and }\,e^{2r_2-1-e^{r_2-1}}>1.$$ Let $C$ be an invariant smooth curve of the system \eqref{gx0}-\eqref{gy0} and $M$ be any compact subset of $X$, then $m_2\left(EF_C\bigcap M\right)=0$ where $m_2$ is a Lebesgue measure in $\mathbb R^2$.
\end{lemma}
\textbf{Remark:} Lemma \ref{l1:pi} and \ref{l2:in_pi} {are a direct corollary from} Theorem 2.2 and 3.3 in Kang (submitted to JDEA).

\section{Sufficient conditions for persistence}\label{sec:ce}
In this section we investigate sufficient conditions for the extinction of one species and the persistence of the other species 
{in} system \eqref{gx0}-\eqref{gy0}. Let $D_\epsilon$ be the set defined in Lemma \ref{l1:pi} and denote $\mathring{D}_\epsilon$ as the 
interior of ${D}_\epsilon$. We can obtain sufficient conditions {for} the extinction of one species by using a Lyapunov function 
$V: \, \mathring{D}_\epsilon\rightarrow \mathbb R_+$ where $V(x,y)= x^c y^d$ and $c, d$ are some {constants}. 
In addition, we give a sufficient condition on the persistence of species $y$ by applying Theorem 2.2 and its corollary of 
Hutson (Hutson 1984) through defining an average Lyapunov function $P(x,y)=y$ in the compact positively invariant region $D_\epsilon$. 
Now we are going to give detailed proof of the following theorem:
 \begin{theorem}\label{th1:ce}[Persistence of One Species]\begin{enumerate}
\item If $r_1>r_2>0$, then the system \eqref{gx0}-\eqref{gy0} has global stability at $(r_1,0)$, i.e., for any initial condition $\xi_0=(x_0,y_0)\in \{(x,y)\in X: x_0>0\}$, we have
$$\lim_{n\rightarrow \infty} H^n(\xi_0)=\lim_{n\rightarrow \infty} H^n(x,y)=(r_1,0).$$
\item If $0<r_1<r_2$, then the species $y$ is {uniformly persistent} in $X$, i.e., {there exists a positive number $\delta>0$ such that for any initial condition $\xi_0=(x_0,y_0)\in \{(x,y)\in X: y>0\}$, we have}
$$\liminf_{n\rightarrow\infty}y_n\geq \delta$$ where $(x_n,y_n)=H^n(\xi_0)$.
Moreover, if $\frac{e^{2r_2-1-e^{r_2-1}}}{r_1}>1$, then the species $x$ goes to extinct for any $\xi_0=(x_0,y_0)\in \{(x,y)\in X: y>0\}$, i.e.,
$$\lim_{n\rightarrow\infty}x_n=0.$$
\end{enumerate}
\end{theorem}
\begin{proof}
According to Lemma \ref{l1:pi}, any point in $X$ is attracted to the compact positively invariant set $D_\epsilon$ for any $\epsilon \in (0, r^m]$. Therefore, we can restrict the dynamics of \eqref{gx0}-\eqref{gy0} to $D_\epsilon$.

If $r_1>r_2>0$, define $V(x,y)=x^{-r_1}y$, then
$$\frac{V\left(H(x,y)\right)}{V(x,y)}=r_1^{-r_1}(x+y)^{r_1}e^{r_2-x-y}.$$
Let $f(u)=r_1^{-r_1} u^{r_1}e^{r_2-u}$. Since $f'(u)=r_1^{-r_1}u^{r_1-1}(r_1-u )e^{r_2-u}$, we can conclude that the maximum value of $f(u)$ achieves at $u=1$, i.e., $$\max_{\epsilon\leq u\leq K} \{f(u)\}=f(r_1)=(\frac{r_1}{r_1})^{r_1} e^{r_2-r_1}<1\mbox{ where }K>\max\{r_1, r_2, e^{r_2-1},1\}.$$ According to Lemma \ref{l1:pi}, we know that $D_\epsilon$ is positively invariant and attracts all points in $X$. Therefore, any point in the region $\{(x_0,y_0)\in D_\epsilon:\, x_0>0\}$ has the following two situations\begin{enumerate}
\item If $y_0=0$, then
 $V\left(H^n(x_0,y_0)\right)=V\left((x_n,y_n)\right)=0$ or
 \item If \modifyb{$x_0,\,y_0>0$}, then
\bea
\frac{V\left(H(x_0,y_0)\right)}{V(x_0,y_0)}&\leq& \max_{(x,y)\in D_\epsilon}\{r_1^{-r_1}(x+y)^{r_1}e^{r_2-x-y}\}\\
&=&\max_{\epsilon\leq u\leq K}\{r_1^{-r_1}u^{r_1}e^{r_2-u}\}=(\frac{r_1}{r_1})^{r_1} e^{r_2-r_1}<1.\eea
\end{enumerate}Thus,
$$\frac{V\left(H^n(x_0,y_0)\right)}{V(x_0,y_0)}=\frac{V\left(H^n(x_0,y_0)\right)}{V^{n-1}(x_0,y_0)}\cdot\cdot\cdot\frac{V\left(H(x_0,y_0)\right)}{V(x_0,y_0)}=\left[(\frac{r_1}{r_1})^{r_1} e^{r_2-r_1}\right]^n\rightarrow 0 \mbox{ as } n\rightarrow \infty.$$ Therefore, the positively invariant property of $D_\epsilon$ implies that
$$\lim_{n\rightarrow\infty} x_n^{-r_1}y_n=0.$$Therefore, $$\lim_{n\rightarrow\infty}y_n=0 \mbox{ and }\liminf_{n\rightarrow\infty}x_n\geq\epsilon.$$
This indicates that
$$\lim_{n\rightarrow\infty} x_{n+1}=\lim_{n\rightarrow\infty} \frac{r_1x_n}{x_n+y_n}=\lim_{n\rightarrow\infty} \frac{r_1}{1+y_n/x_n}=r_1.$$
Therefore, if $r_1>r_2>0$, then the system \eqref{gx0}-\eqref{gy0} has global stability at $\xi^*=(r_1,0)$. The first part of Theorem \ref{th1:ce} holds.

\indent If $r_2>r_1>0$, then the omega limit set of $S_x=\{(x,0): x>0\}$ is $\xi^*$, i.e., $\omega(S_x)=\{\xi^*\}$.  The external Lyapunov exponent of $S_x$ is $e^{r_2-r_1}>1$, therefore, it is transversal unstable.  According to Lemma \ref{l1:pi}, for any $0<\epsilon\leq r^m$, $D_\epsilon$ attracts all points in $X$ . Thus, the uniform persistence of species $y$ follows from Theorem 2.2 and its corollary of Hutson (Hutson 1984) by defining a Lyapunov function $P(x,y)=y$ on the compact positively invariant region $D_\epsilon$, i.e., there exists a positive number $\delta>0$ such that for any $y_0>0$, we have
$$\liminf_{n\rightarrow\infty}y_n>\delta.$$
If, in addition, $r_1<e^{2r_2-1-e^{r_2-1}}$, then we can define a Lyapunov function as $V(x,y)=xy^{-1}$, then we have
$$\frac{V\left(H(x,y)\right)}{V(x,y)}=\frac{r_1}{(x+y)e^{r_2-(x+y)}}.$$
Now choose $\epsilon=\min\{r_1,r_2, e^{2r_2-1-e^{r_2-1}},r_1e^{r_2-r_1}\}$, then $\epsilon=r_1$ since $r_2<r_1<e^{2r_2-1-e^{r_2-1}}$. Therefore, any point $(x_0,y_0)\in D_{r_1}$ satisfies $r_1\leq x+y\leq e^{r_2-1}$ and will stay in $D_{r_1}$ for all future time. This implies that for any point $(x_0,y_0)$ in $D_{r_1}$ with $y_0>0$, we have
\bea
\frac{V\left(H(x_0,y_0)\right)}{V(x_0,y_0)}&\leq& \max_{(x,y)\in D_{r_1}}\{\frac{r_1}{(x+y)e^{r_2-(x+y)}}\}=\frac{r_1}{\min_{(x,y)\in D_{r_1}}\{(x+y)e^{r_2-(x+y)}\}}\\
&=&\frac{r_1}{\min_{r_1<u<e^{r_2-1}}\{u e^{r_2-u}\}}=\frac{r_1}{e^{2r_2-1-e^{r_2-1}}}<1\eea
Hence, $\lim_{n\rightarrow\infty}x_n=0.$  Now if $(x_0,y_0)\in X\setminus D_{r_1}$, then according to Lemma \ref{l1:pi}, $(x_0,y_0)$ will either enter $D_{r_1}$ in some finite time or converge to $(r_1,0)$.  Now we consider the following two cases for any initial condition $(x_0,y_0)\in X\setminus D_{r_1}$ with $y_0>0$:
\begin{enumerate}
\item If $x_0=0$, then $x_n=0$ for all positive integer $n$;
\item If $x_0>0$, then $(x,y)$ will not converge to $(r_1,0)$ since the equilibrium point $(r_1,0)$ is a saddle and transversal unstable when $r_2>r_1$, therefore,  $(x,y)$ will
enter $D_{r_1}$ in some finite time.
\end{enumerate}
Thus, the condition $r_1<r_2$ and $r_1<e^{2r_2-1-e^{r_2-1}}$ guarantees that
$$\lim_{n\rightarrow\infty}x_n=0.$$
Therefore, the second part of Theorem \ref{th1:ce} holds.
\end{proof}
\noindent \textbf{Remark:} The first part of Theorem \ref{th1:ce} can be considered as a special case of rational growth rate dominating {exponential (Franke and Yakubu 1991) which} states that if species \modifyb{$x$ with rational growth rate can invade species $y$ with exponential growth rate at species $y$'s fixed point, i.e., $(0, r_2)$ is transversal unstable}, then the exponential species goes extinct irrespective of the initial population sizes. The second part of Theorem \ref{th1:ce} shows that the exponential species can persistent whenever $(r_1,0)$ is \modifyb{transversal} unstable (i.e., $r_2>r_1$). However, the rational species may not go extinct unless $r_1< e^{2 r_2-1-e^{r_2-1}}$. In fact, simulations (e.g., Figure \ref{basind}) suggest that two species of the system \eqref{gx0}-\eqref{gy0} may coexist for almost every initial conditions in $X$ under certain conditions. This {point will be illustrated with greater details in} the next section.


\section{Coexistence of two species}\label{sec:rp}
In this section, we {give sufficient conditions for the} existence of the interior period-2 orbits and its local stability for the system \eqref{gx0}-\eqref{gy0} as the following theorem  states:
\begin{theorem}\label{th2:bifurcation}[Sufficient conditions on the existence of interior period-2 orbits] If $r_2>2$, then the Ricker map $y_{n+1}=y_n e^{r_2-y_n}$ has period two orbits $\{y_1,y_2\}$ where $0<y_1<r_2<y_2$ and $y_1+y_2=2r_2$.  {The system \eqref{gx0}-\eqref{gy0} has an interior period-2 orbit $P^i_2=\{(x_1^i,y_1^i),(x_2^i,y_2^i)\}$ where
\bae\label{p2}
\begin{array}{lcl}
x_1^i=\frac{s_1(s_1 e^{r_2-s_1}-s_2)}{s_1e^{r_2-s_1}-r_1},&& \,\,y_1^i=\frac{s_1(s_2-r_1)}{s_1e^{r_2-s_1}-r_1}\\
x_2^i=\frac{r_1x_1^i}{s_1},&& \,\,y_2^i=y_1^i e^{r_2-s_1}\\
s_1=x_1^i+y_1^i=r_2-\sqrt{r_2^2-r_1^2},&& \,\, s_2=x_2^i+y_2^i=r_2+\sqrt{r_2^2-r_1^2}
\end{array}
\eae 
if one of the follows holds
\begin{enumerate}
\item$s_1 e^{r_2-s1}>s_2$, or
\item$r_2-\sqrt{r_2^2-r_1^2}>y_1$, or
\item$2\leq r_1<r_2<2.5$ and $r_1>r_2-\frac{(\frac{r_2-2}{0.26})^2}{2r_2}$, or
\item$2.085\leq r_1\leq r_2\leq 2.5$
\end{enumerate} In particular, (4) implies (3); (3) implies (2) and (2) implies (1).
Moreover, if $r_1=2$ and $\delta=r_2-r_1=r_2-2$ is small enough, then $P^i_2$ is locally asymptotically stable.}
\end{theorem}
\begin{proof}If $r_2>r_1>0$, then $$r_2-r_1-\sqrt{r_2^2-r_1^2}=\sqrt{r_2-r_1}\left(\sqrt{r_2-r_1}-\sqrt{r_2+r_1}\right)<0,$$ thus we have the following inequalities:
$$s_1=r_2-\sqrt{r_2^2-r_1^2}<r_1<r_2<s_2=r_2+\sqrt{r_2^2-r_1^2}.$$
Therefore, from \eqref{p2}, we find that $s_1e^{r_2-s_1}-s_2>0$ is a sufficient condition for the existence of $P^i_2$.

Notice that the Ricker map $y_{n+1}=y_n e^{r_2-y_n}$ goes through period-doubling two bifurcation at $r=2$, thus if $r_2>2$, the Ricker map has a period-2 orbit $\{y_1,y_2\}$ where
$$0<y_1<r_2<y_2,\, y_2=y_1 e^{r_2-y_1}\,\mbox{ and }\,y_1+y_2=2r_2.$$
Since $s_1+s_2=2 r_2$,  then from the graphic representation (see Figure \ref{fig:p2}), we can see that
$$s_1e^{r_2-s_1}-s_2>0\mbox{ whenever }y_1<s_1<r_2.$$
Therefore, the condition $s_1=r_2-\sqrt{r_2^2-r_1^2}>y_1$ is a sufficient condition for $s_1e^{r_2-s_1}-s_2>0$, therefore, it is a sufficient condition for the existence of $P^i_2$.
\begin{figure}[ht]
\begin{center}
\includegraphics[width=100mm]{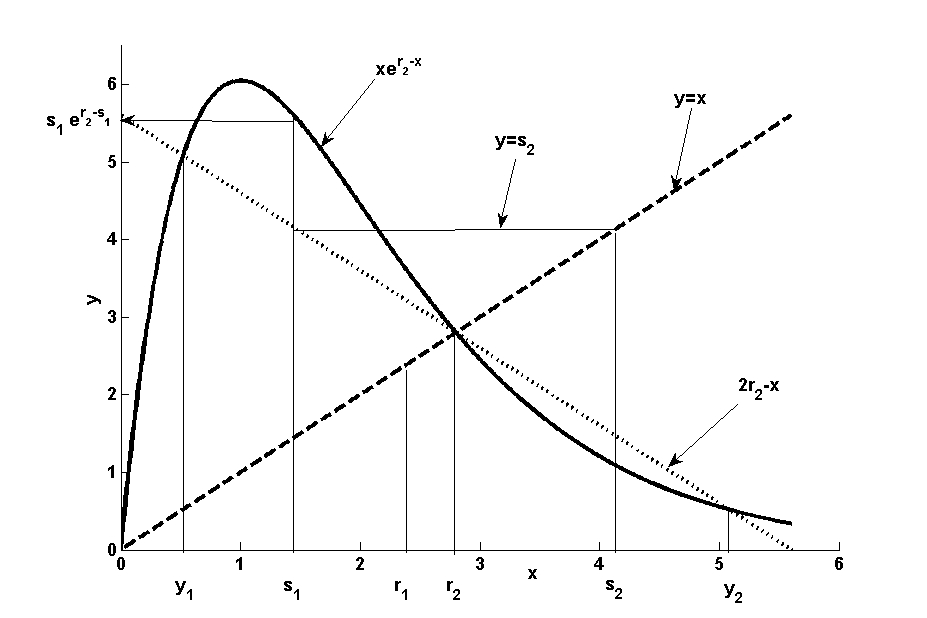}
 \caption{The location between $y_i,s_i, i=1,2$ and $s_1e^{r_2-s_1}$. The solid line is $f(y)=ye^{r_2-y}$; the dashed line is $f(y)=y$; the dot line is $f(y)=2 r_2-y$.}
 \label{fig:p2}
 \end{center}
\end{figure}

Let $a=\sqrt{r_2^2-r_1^2}$, then we have the following equivalent relationships:
\be\label{sq}s_1 e^{r_2-s_1}>s_2\modifyb{\iff}(r_2-a)e^a>r_2+a\modifyb{\iff} r_2>\frac{2a(e^a+1)}{e^a-1}=a+\frac{2a}{e^a-1}\modifyb{\iff} r_2-a>\frac{2a}{e^a-1}\ee
Thus we find that $r_2-a>\frac{2a}{e^a-1}$ implies $s_1e^{r_2-s_1}-s_2>0$.

If $2\leq r_1<r_2<2.5$, then $$0<a=\sqrt{r_2^2-r_1^2}=\sqrt{r_2-r_1}\sqrt{r_2+r_1}<\sqrt{0.5}\sqrt{(2.5+2.5)}=\frac{\sqrt{10}}{2}.$$
Notice that $h(a)=\frac{2a}{e^a-1}$ is a decreasing convex function with respect to $a$, thus
$$h(a)\leq k(a)=2-\frac{2-\frac{2\frac{\sqrt{10}}{2}}{e^\frac{\sqrt{10}}{2}-1}}{\frac{\sqrt{10}}{2}} a$$ where $k(a)$ is a straight line going through $(0,2)$ and $\left(\frac{\sqrt{10}}{2}, h(\frac{\sqrt{10}}{2})\right)$. Since $\frac{2-\frac{2\frac{\sqrt{10}}{2}}{e^\frac{\sqrt{10}}{2}-1}}{\frac{\sqrt{10}}{2}}>0.74$, therefore,
$$2-0.74 a\geq2-\frac{2-\frac{2\frac{\sqrt{10}}{2}}{e^\frac{\sqrt{10}}{2}-1}}{\frac{\sqrt{10}}{2}} a\geq\frac{2a}{e^a-1}, \mbox{ for all } 0<a<\frac{\sqrt{10}}{2}.$$
Hence, from \eqref{sq}, we can conclude that $r_2-a>2-0.74 a$ implies $r_2-a>\frac{2a}{e^a-1}$, therefore it implies $s_1e^{r_2-s_1}-s_2>0$. Notice the following equivalent relationships,

\be\label{sq1}r_2-a>2-0.74 a\modifyb{\iff}  a<\frac{r_2-2}{0.26}\modifyb{\iff} r_2^2-r_1^2<(\frac{r_2-2}{0.26})^2\modifyb{\iff} r_1>r_2-\frac{(\frac{r_2-2}{0.26})^2}{r_2+r_1},\ee
therefore, we can conclude that  $r_1>r_2-\frac{(\frac{r_2-2}{0.26})^2}{r_2+r_1}$ implies $r_2-a>2-0.74 a$, therefore, it implies $r_2-a>\frac{2a}{e^a-1}$, therefore, it implies $s_1e^{r_2-s_1}-s_2>0$, therefore, it implies the existence of $P^i_2$.

Since $2\leq r_1<r_2<2.5$, then $r_2+r_1<2 r_2\leq 5$, thus
$$r_1>r_2-\frac{(\frac{r_2-2}{0.26})^2}{5}\Rightarrow r_1>r_2-\frac{(\frac{r_2-2}{0.26})^2}{2r_2}\Rightarrow r_1>r_2-\frac{(\frac{r_2-2}{0.26})^2}{r_2+r_1}.$$
Therefore,  $r_1>r_2-\frac{(\frac{r_2-2}{0.26})^2}{2r_2}$ implies $r_1>r_2-\frac{(\frac{r_2-2}{0.26})^2}{r_2+r_1}$, therefore, it implies the existence of $P^i_2$.

Notice that $$r_2-\frac{(\frac{r_2-2}{0.26})^2}{5}=-2.958579881(r_2-2.169000001)^2+2.08450001\leq 2.085,$$ hence we can conclude that
$$2.085\leq r_1<r_2\leq 2.5$$ implies $r_1>r_2-\frac{(\frac{r_2-2}{0.26})^2}{r_2+r_1}$, therefore, it implies the existence of $P^i_2$.

So far, we have shown the first part of Theorem \ref{th2:bifurcation}. Now we are going to see that local stability of $P^i_2$. Let $r_1=2$ and $\delta=r_2-r_1=r_2-2$, then we have
$$s_1=r_2-\sqrt{r_2^2-r_1^2}=2+\delta-\sqrt{\delta(4+\delta)}\mbox{ and }s_2=2+\delta+\sqrt{\delta(4+\delta)}.$$ Thus if $\delta$ is small enough, then
$$s_1e^{r_2-s_1}-s_2=(10/3)\delta^{3/2}-(10/3)\delta^2+(41/12)\delta^{5/2}+O(\delta^3)>0.$$
Therefore, from the proof for the first part of Theorem \ref{th2:bifurcation}, we can conclude that the system \eqref{gx0}-\eqref{gy0} has an interior period-2 orbit $P^i_2$ when $r_1=2$ and $\delta=r_2-r_1=r_2-2$ is small enough. The local stability of $P^i_2$ is determined by \modifyb{the eigenvalues of the product of the Jacobian matrices along the periodic-2 orbit} which can be represented as follows:
\be\label{J}
J\vert_{P^i_2}= \left( \begin {array}{cc}\frac{y_1^iy_2^i}{r_1^2}+\frac{r_1x_1^iy_2^ie^{r_2-s_2}}{s_1^2}&\,\,\,-\frac{y_1^ix_2^i}{r_1^2}+\frac{r_1x_1^i\left(-1+y_2^i\right)e^{r_2-s_2}}{s_1^2}\\\noalign{\medskip}- \frac{r_1y_1^iy_2^ie^{r_2-s_1}}{s_2^2} +y_2^i\left(y_1^i-1\right)&\,\,\, -\frac{r_1y_1^ix_2^ie^{r_2-s_1}}{s_2^2}+\left(y_1^i-1\right)\left(y_2^i-1\right)\end {array} \right).\ee
If $\delta$ is small enough, then the trace and determinant of \eqref{J} can be approximated by
$$det(J)+1= 2-\frac{8\delta}{3}+\frac{49\delta^2}{30}+O(\delta^3)\,\,\mbox{ and }\,\,trace(J)= 2-\frac{8\delta}{3}+\frac{3\delta^2}{10}+O(\delta^3).$$
By the Jury test (p57 in Edelstein-Keshet 2005), we see that $P_2^i$ is locally asymptotically stable if
\be\label{sJ}2 > 1 + det(J) = 2-\frac{8\delta}{3}+\frac{49\delta^2}{30}+O(\delta^3)> \vert trace(J)\vert = \vert 2-\frac{8\delta}{3}+\frac{3\delta^2}{10}+O(\delta^3)\vert\ee which is true when $\delta$ is small enough.

Therefore, the statement of Theorem \ref{th2:bifurcation} holds.
\end{proof}

\textbf{Remark:} Theorem \ref{th2:bifurcation} provides a sufficient condition on the existence of the interior period two orbit and their stability. Numerical simulations suggest that the system \eqref{gx0}-\eqref{gy0} has an interior period two orbit whenever $2\leq r_1<r_2<2.5$. In the case that $r_1=2$ and $r_2=2+\delta$, the interior period two orbit $P_2^i$ is locally asymptotically stable whenever $\delta<0.95$ (see Figure \ref{fig:p2i}-\ref{fig:sp2i}).
\begin{figure}[ht]
\begin{center}
\includegraphics[width=100mm]{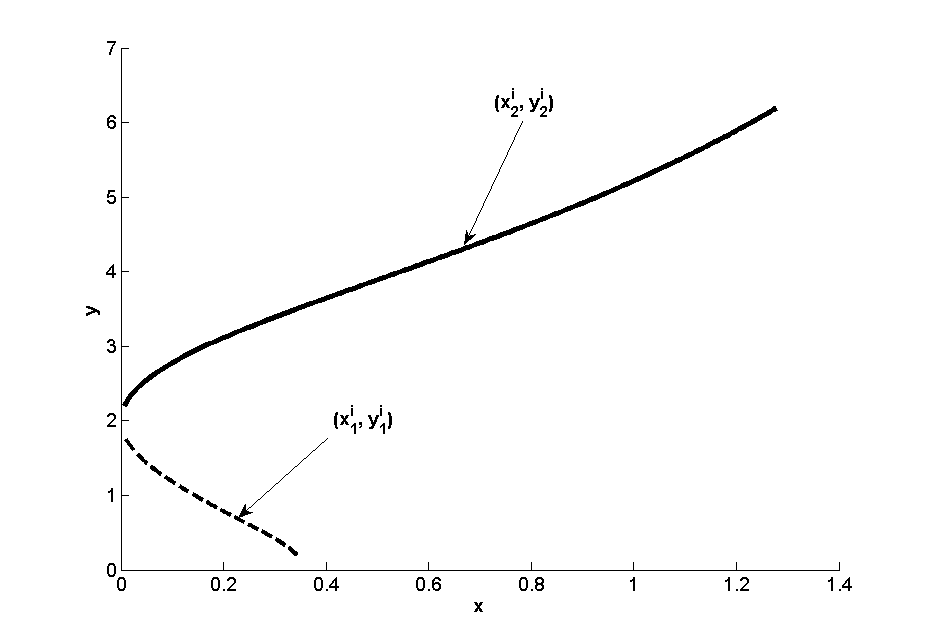}
 \caption{Interior period two orbit $P_2^i$ of the system \eqref{gx0}-\eqref{gy0} when $r_1=2, r_2=2+\delta$ and $\delta$ is varying from 0 to 2. The solid line is $(x_2^i,y_2^i)$ and the dashed line is $(x_1^i,y_1^i)$.}
 \label{fig:p2i}
 \end{center}
\end{figure}

\begin{figure}[ht]
\begin{center}
\includegraphics[width=100mm]{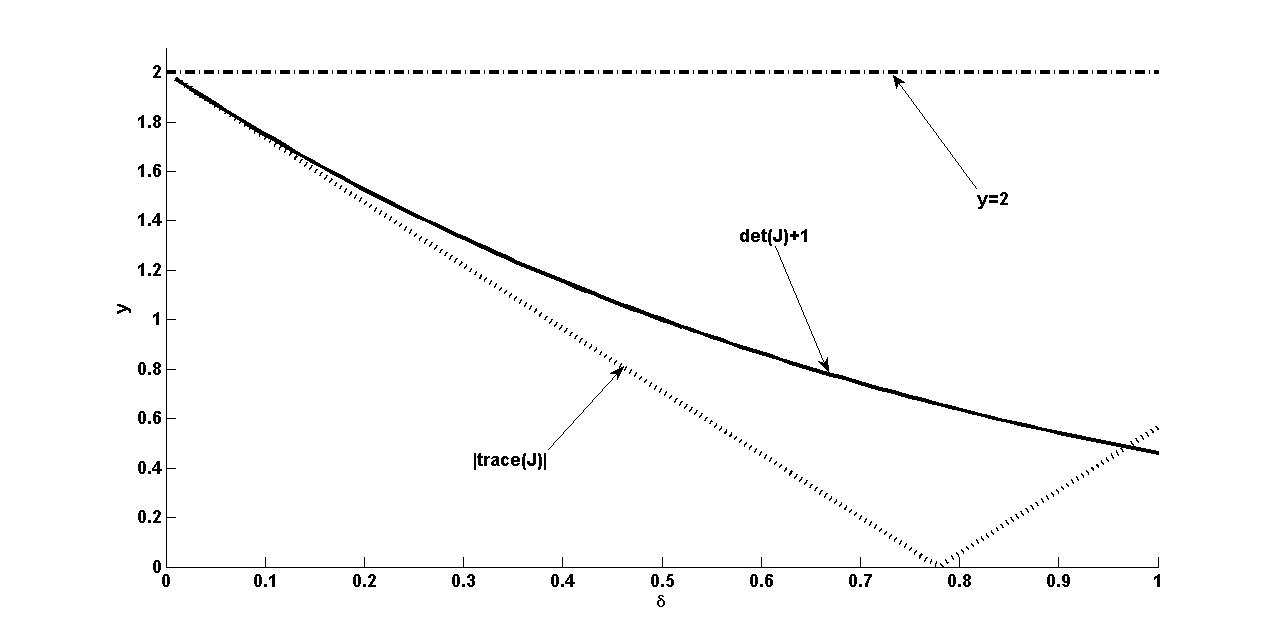}
 \caption{The stability of the interior period-2 orbit $P_2^i$ of the system \eqref{gx0}-\eqref{gy0} when $r_1=2, r_2=2+\delta$ and $\delta$ is varying from 0 to 1. The solid line is $det(J)+1$; the dashed-dot line is constant 2 and the dot line is $\vert trace(J)\vert$. This figure indicates that $P_2^i$ is locally asymptotically stable when $r_1=2$ and $2< r_2=2+\delta<2.95.$ }
 \label{fig:sp2i}
 \end{center}
\end{figure}
\begin{lemma}\label{l3:measure0}[Pre-images of Heteroclinic orbit]
Assume that the system \eqref{gx0}-\eqref{gy0} satisfies Condition \textbf{C1} and \textbf{C3}, then there exists a smooth invariant curve $C$ that connecting $\xi^*$ to $\eta^*$.  Denote $EF_C$ as the collection of all ranks pre-images of $C$, then $m_2(EF_C\cap M)=0$, where $M$ is any compact subset of $X$ and $m_2$ is a Lebesgue measure in $\mathbb R^2_+$.\end{lemma}
\begin{proof}
First we show that $C$ is a smooth curve connecting $\xi^*=(r_1,0)$ to $\eta^*=(0,r_2)$. Let $\omega^u_l(\xi^*)$ be the local unstable manifold of $\xi^*$ and $\omega^s_l(\eta^*)$ be the local stable manifold of $\eta^*$. Since the map $H$ is smooth, then according to stable manifold theorem (Theorem D.1 in Appendix, Elaydi 2005), we can conclude that both  $\omega^u_l(\xi^*)$ and $\omega^s_l(\eta^*)$ are smooth curves. Since \eqref{gx0}-\eqref{gy0} satisfies Condition \textbf{C3}, then there exists some positive integer $k$ such that $H^k\left(\omega^u_l(\xi^*)\right)$ is smoothly connected with $\omega^s_l(\eta^*)$. Thus $C$ is a smooth invariant curve with $\xi^*, \eta^*$ as its two end points. Then according to Lemma \ref{l2:in_pi}, the statement holds
\end{proof}

 \subsection{Persistence of species $x$ in new space $\tilde{X}$}
Let $$S_x=\{(x,0)\in D_\epsilon\},S_y=\{(0, y)\in D_\epsilon\},S=S_x\bigcup S_y=\{(x,y)\in D_\epsilon: xy=0\}.$$  Let $C$ be the closure of all heteroclinic orbits connecting from $\xi^*$ to $\eta^*$ and denote $EF_{C}$ as the collection of all rank pre-images of $C$. Then we can conclude that $E=\overline{EF_C\cap D_\epsilon}$ is compact and forward invariant. Define the following new spaces
 $${\tilde{S}=S\bigcup E, \tilde{X}=X\setminus \tilde{S}}.$$ Then both $\tilde{S}$ and $\tilde{X}$ is positively invariant. In addition, $\tilde{S}$ is compact since both $S$ and $E$ are compact.

The rest of this section, we assume that the system \eqref{gx0}-\eqref{gy0} satisfies Condition \textbf{C1}-\textbf{C3}. We will prove the following theorem:
\begin{theorem}\label{th3:rp}[Relative Permanence] Assume that the system \eqref{gx0}-\eqref{gy0} satisfies \textbf{C1-C3}. Denote $EF_C$ as the collection of all pre-images of the heteroclinic curve $C$. Then there there exists a compact interior attractor in $R^2_+$ that attracts all points in the interior of $X$ except points in $EF_C$.  In particular, the interior attractor attracts \modifyb{almost every point with respect to a Lebesgue measure} in $\mathbb R^2$ of any compact subset $M$ in the interior of $X$, i.e., $m_2\left(EF_C\bigcap M\right)=0$ where $m_2$ is a Lebesgue measure in $\mathbb R^2$.
\end{theorem}
\begin{proof}
We use the following three main steps to prove the statement.  We provide the detailed proof of the first two steps in the Appendix and the remaining proof here.
\begin{enumerate}
\item $C$ is a uniform weak repeller with respect to $\tilde{X}$, i.e., there exists some $b>0$, for any $\xi\in \tilde{X}$, we have $$\limsup_{n\rightarrow\infty} d(H^n(\xi), C)>b.$$ The detailed proof of this part has been shown in the Appendix (Lemma \ref{repellingC}). This implies that any point in $\tilde{X}$ is going to be away from $C$ in some distance in some future time even if the point is very close to $C$, .
\item Species $x$ is uniformly weakly persistent in $\tilde{X}$, i.e., there exists some $\delta>0$, such that for any initial condition $\xi_0=(x_0,y_0)\in \tilde{X}$, the system has $$\limsup_{n\rightarrow\infty} x_n>\delta.$$The detailed proof of this part has been shown in the Appendix (Lemma \ref{uwr}). This implies that any point in $\tilde{X}$ is going to be away from $\tilde{S}$ in some distance in some future time even if the point is very close to $\tilde{S}$.
\item Species $x$ is uniformly persistent in $\tilde{X}$, i.e., there exists some $\epsilon>0$, such that for any initial condition $\xi=(x,y)\in\tilde{X}$, the system has $$\liminf_{n\rightarrow\infty} x_n >\epsilon.$$
\end{enumerate}
Now we will show that the last step. Define a continuous and not identically zero persistent function $\rho(\xi)=d(\xi,\tilde{S})$ where $\xi\in X$. Then by the definition of the persistent function $\rho$, we have
\[\tilde{S}=\{\xi\in X: \rho(H^t(\xi))=0, \forall t\geq 0\}\] is nonempty, closed and positively invariant.  In addition, the system $H$ satisfies the following two conditions:
\begin{enumerate}
\item There exists no bounded total trajectory $\phi$ such that $\rho(\phi(0))=0, \rho(\phi(-r))=0$ and $\rho(\phi(t))>0$ for some positive integer $r,t$.
\item $H$ has a compact attractor $D_\epsilon$ which attracts all points in $X$.
\end{enumerate}

From Lemma \ref{uwr}, we know that species $x$ is uniformly weakly persistent in $\tilde{X}$. Thus by applying {Theorem 5.2 Smith \& Thieme 2010}, we can conclude that species $x$ is uniformly persistent in $\tilde{X}$, i.e., there exists some $\epsilon>0$, such that for any initial condition $\xi=(x,y)\in\tilde{X}$, the system has $$\liminf_{n\rightarrow\infty} x_n >\epsilon.$$

Notice that species $y$ is uniformly persistent in $X$ whenever $r_2>r_1$ according to Theorem \ref{th1:ce}. Thus, species $y$ is also uniformly persistent in $\tilde{X}$ since $\tilde{X}$ is a positively invariant subset of $X$. Therefore, based on the argument above, we can conclude that there exists some positive constant $\epsilon>0$, such that for any initial condition taken in $\tilde{X}$, we have $$\liminf_{n\rightarrow\infty}\min\{x_n, y_n\}>\epsilon.$$
Hence there there exists a compact interior attractor that attracts all points in the interior of $X$ except points in $EF_C$.

Let $M$ be any compact subset of the interior of $X$. Then any initial condition $\xi_0$ taken in $M$ will enter $D_\epsilon$ in some future time through the following two cases
\begin{enumerate}
\item $\xi_0\in EF_C$ which will enter $C$ in some finite time.
\item $\xi_0\in \tilde{X}$ which is attracted to the interior compact attractor.
\end{enumerate}
Since there are only Lebesgue measure zero of points in $M$ that belong to $EF_C$, therefore, according to Lemma \ref{l2:in_pi}, we can conclude that $m_2\left(EF_C\bigcap M\right)=0$ for any compact subset $M$ in the interior of $X$.
\end{proof}
\section{Discussion and future work}\label{sec:con}
In this article, we study the global dynamics of the system \eqref{gx0}-\eqref{gy0}. We give sufficient conditions {for} the uniform persistence of one species and the existence of locally asymptotically stable interior period-2 orbits for this system. We also show that {for a} certain parameter range, the system \eqref{gx0}-\eqref{gy0} is \emph{relative permanent}, i.e., there exists a compact interior attractor that attracts almost all points in $X$. Numerical simulations strongly {suggest} that this compact interior attractor is the locally asymptotically stable interior period-2 orbit $P_2^i$ and its {basin} of attractions consists of a infinite union of connected regions that are separated by all pre-images of the heteroclinic curve $C$ (see Figure \ref{basind}).

The results that we obtained in Theorem \ref{th1:ce} {are} a special case for model \eqref{gx}-\eqref{gy} when $a=0$. Our Theorem \ref{th2:bifurcation} can be extended to the general model \eqref{gx}-\eqref{gy} when $a>0$. If $r_2>2$ and $ r_1^2>2 a r_2$, the explicit expressions of the interior periodic-2 orbit $\{(x_1^i,y^i_1), (x^i_2,y^i_2\}$ of the system \eqref{gx}-\eqref{gy} can be found as
\bea
x_1^i=\frac{\left(a+s_1\right)\left(s_2-s_1e^{r_2-s_1}\right)}{r_1-\left(a+s_1\right)e^{r_2-s_1}},&& \,\,y_1^i=\frac{r_1s_1-s_2\left(a+s_1\right)}{r_1-\left(a+s_1\right)e^{r_2-s_1}}\\
x_2^i=\frac{\left(a+s_2\right)\left(s_1-s_2e^{r_2-s_2}\right)}{r_1-\left(a+s_2\right)e^{r_2-s_2}},&& \,\,y_2^i=\frac{r_1s_2-s_1\left(a+s_2\right)}{r_1-\left(a+s_2\right)e^{r_2-s_2}}.
\eea where
$$s_1=x_1^i+y_1^i=r_2+\sqrt{(r_2+a)^2-r_1^2} ,\,\, s_2=x_2^i+y_2^i=r_2-\sqrt{(r_2+a)^2-r_1^2}.$$
This interior periodic-2 orbit can have local stability for {a} certain range of parameters' values. For instance, if
$$r_1=2.1, a=0.1 \mbox{ and } r_2=2.5,$$ then the system \eqref{gx}-\eqref{gy} has locally stable interior periodic-2 orbit
$$(x_1^i,y_1^i)=(0.17,0.80) \mbox{ and } (x_2^i, y_2^i)=(0.33, 3.70)$$ along which the eigenvalues of the product of the Jacobian matrices are 0.11 and -0.24. Moreover, numerical simulation suggests follows
\begin{enumerate}
\item There exits a heteroclinic orbit $C$ connecting $\xi^*$ to $\eta^*$ (see Figure \ref{hc_orbit});
\begin{figure}[ht]
\centering
\includegraphics[width=100mm]{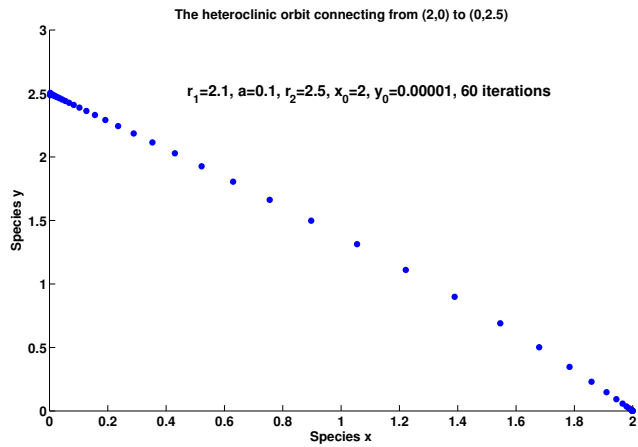}
 \caption{The heteroclinic orbit of the system \eqref{gx}-\eqref{gy} when $r_1=2.1, a=0.1,r_2=2.5$.}
 \label{hc_orbit}
\end{figure}
\item The basin of attraction of the interior periodic-2 orbit $P_2^i$ is all points in the interior of $\mathbb R^2_+$ except a Lebesgue measure zero set in $\mathbb R^2$ which is a collection of all pre-images of the heteroclinic curve $C$ (see Figure \ref{basin}).
\begin{figure}[ht]
\centering
\includegraphics[width=100mm]{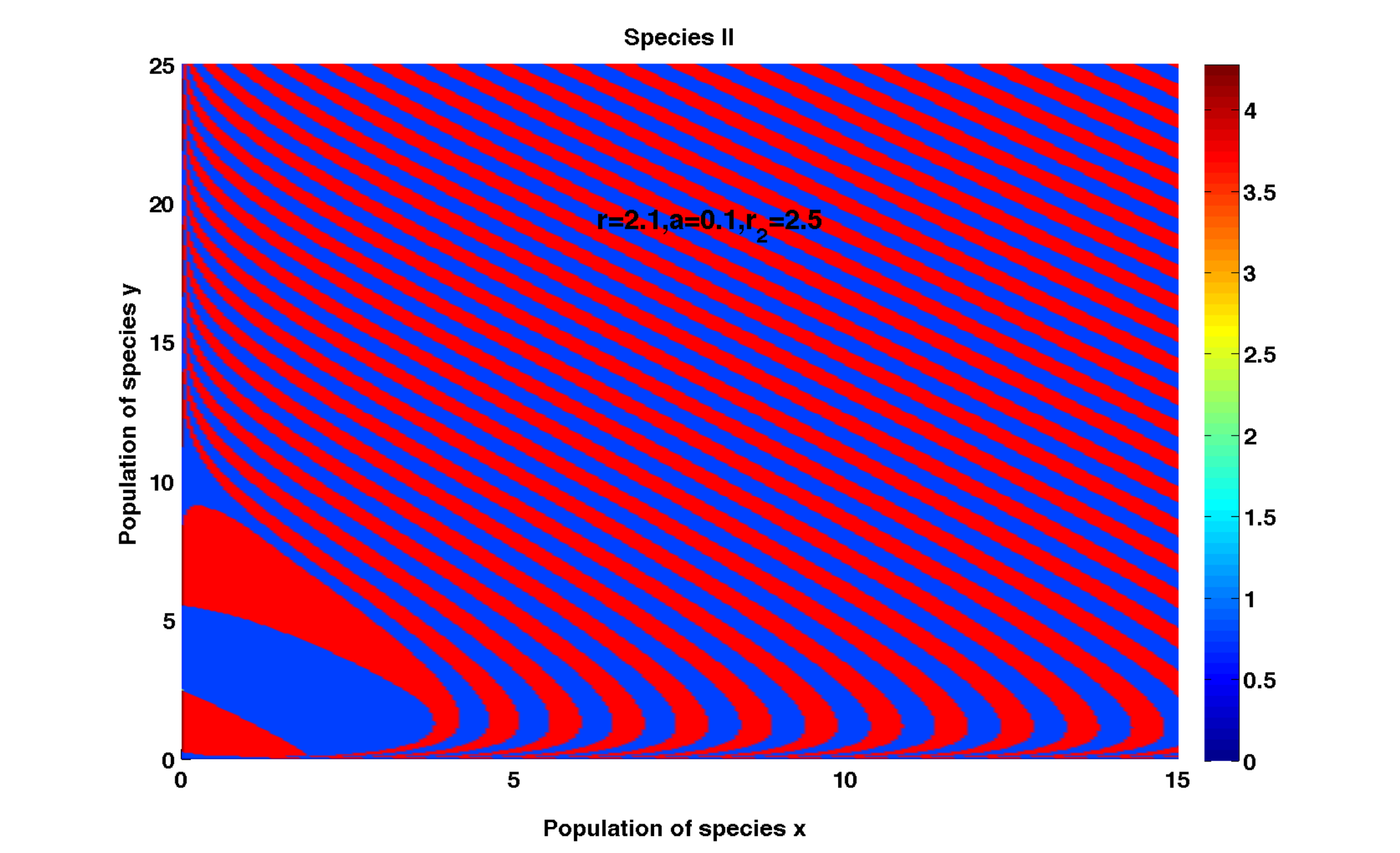}
 \caption{The basin of attraction of the interior period-2 orbit of the system \eqref{gx}-\eqref{gy} when $r_1=2.1, a=0.1,r_2=2.5$ is the open quadrant minus all pre-images of the heteroclinic curve $C$. The latter partition the quadrant into components which are colored according to which of the two periodic points attract points in the component under the second iterate of the map. Given a point in one of the regions, there is a large number $N$, such that the point will be very close to $(x_1^i,y_1^i)$ at the iteration $t$ and will be very close to $(x_2^i,y_2^i)$ at the iteration $t+1$ for all $t>N$.}
 \label{basin}
\end{figure}
\end{enumerate}
However, more mathematical techniques {need to be developed in order to obtain results similar to those in Lemma \ref{l2:in_pi}} for the system \eqref{gx}-\eqref{gy} when $a>0$. {This is an area for future study.}

Our results may apply to a two species discrete-time Lotka-Volterra competition model with {stocking where} both species are governed by Ricker's model and one species is being stocked at the constant per capita stocking rate $s_1$ per generation \eqref{c2x}-\eqref{c2y} (Elaydi and Yakubu 2002a\& 2002b). {We may infer} from simulations (see Figure \ref{c2_phaseplane}) that the basin of attraction of the 2-cycle is the infinite union of connected regions that {are} separated by all pre-images of the heteroclinic curve $C$ when $s_1=.5, q_1=1.5,q_2=2.2,p_1=p_2=1$.
\bae\label{c2x}
x_{n+1}&=&x_n \left[s_1+e^{q_1-p_1(x_n+y_n)}\right]\\
\label{c2y}
y_{n+1}&=&y_n e^{q_2-p_2(x_n+y_n)}
\eae
\begin{figure}[ht]
\begin{center}
\includegraphics[width=120mm]{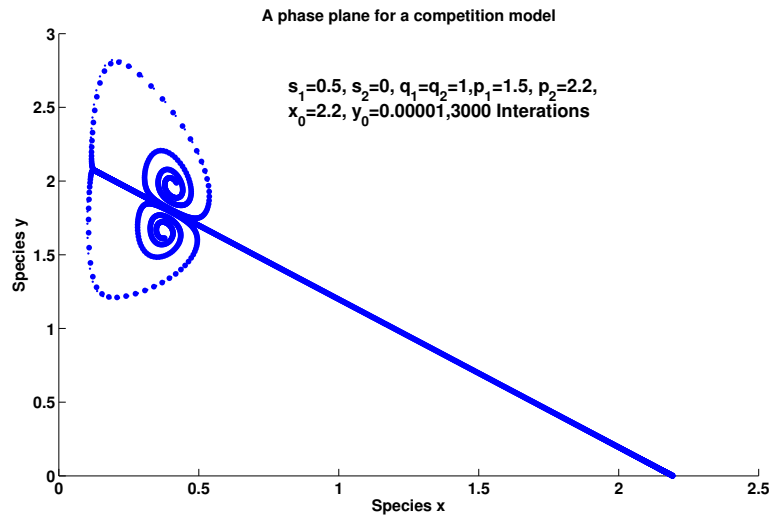}
 \caption{A single forward orbit of the system \eqref{c2x}-\eqref{c2y} starting near the fixed point on $x$-axis.} \label{c2_phaseplane}
 \end{center}
 \end{figure}

The system \eqref{gx}-\eqref{gy} is not the only competition model that has a local stable interior period-2 orbit that attracts all points of $\mathbb R^2_+$ except all pre-images of the heteroclinic curve $C$ that is connecting two nontrivial boundary equilibria. In general, if a discrete two-species competition model satisfies the following conditions (see Figure \ref{gstructure} for a schematic presentation), numerical simulations suggest that it can have an interior attractor that attracts all points of $\mathbb R^2_+$ except all pre-images of the heteroclinic orbit that is connecting two nontrivial boundary equilibria.  It will be our future work to develop more powerful analytic tools to rigorously prove this.
\begin{itemize}
\item \textbf{G1:}The system has only two nontrivial boundary equilibria $(x^*,0)$ and $(0, y^*)$. Moreover, species $y$ is persistent.
\item \textbf{G2:}The omega limit set of $y$-axis is a unique attracting period-2 orbit $M_y=\{\eta_1,\eta_2\}=\{(0,y_1),(0,y_2)\}$ on y-axis, which attracts all points in $y$-axis except a Lebesgue measure zero set. In addition, the external Lyapunov exponent of $M_y$ is greater than 1, i.e., species $x$ can invade species $y$ on $M_y$.
\item \textbf{G3:}There is a heteroclinic orbit connecting the boundary equilibrium $(x^*,0)$ to $(0, y^*)$.
\end{itemize}

\begin{figure}[ht]
 \centering
  \includegraphics[width=120mm]{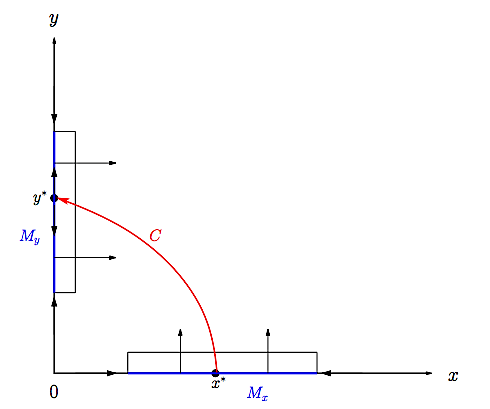}
   \caption{The general structure of dynamics that the basin of attraction of interior attractors (e.g.,the interior periodic-2 orbit) is all points in $\mathbb R^2_+$ except the collection of all pre-images of the heteroclinic curve $C$.}
   \label{gstructure}
   \end{figure}


\appendix
\section{Important Lemmas}

\begin{lemma}\label{repellingC}[Uniform Weak Repeller]
If the system \eqref{gx0}-\eqref{gy0} satisfied satisfies Condition \textbf{C1}-\textbf{C3}. Then there exists some $b>0$, such that $$\limsup_{n\rightarrow\infty} d(H^n(\xi), C)>b\mbox{ for any } \xi\in \tilde{X}. $$\end{lemma}
\begin{proof}
The condition $2<r_2<2.52$ indicates that species $y$ has a unique attracting period-two orbit $$\{\eta_1,\eta_2\}=\{(0,y_1), (0, y_2)\}$$ in its single state and the condition $r_1<r_2$ implies that the boundary equilibrium $\eta^*=(0, r_2)$ is a saddle. By Hartman-Grobman-Cushing theorem \modifyb{(Elaydi 2005)}, there exists some neighborhood $U_{\epsilon_1}(\eta^*)$ of $\eta^*$, such that any point $\eta\in U_{\epsilon_1}(\eta^*)\cap\tilde{X}$ will exit from this neighborhood in some finite time. If we choose $\epsilon$ small enough, then $\eta$ is attracted to a compact neighborhood $$B=\overline{U_\delta(M_y)}=\{\xi\in\tilde{X}: d(\xi, M_y)\leq \delta \}$$ where $M_y=\{\eta_1,\eta_2\}$ in some finite time. Similarly, the condition $0<r_1<r_2$ implies that the boundary equilibrium $\xi^*$ is also a saddle, by Hartman-Grobman-Cushing theorem, there exists some neighborhood $U_{\epsilon_2}(\xi^*)$ of $\xi^*$, such that any point $\xi\in U_{\epsilon_2}(\xi^*)\cap\tilde{X}$ will exit from this neighborhood in some finite time.

Choose $\epsilon=\min\{\epsilon_1,\epsilon_2\}$. Let $K=\overline{C\setminus \left(U_\epsilon(\xi^*)\bigcup U_\epsilon(\eta^*)\right)}$, then $K$ is a compact subset of $C$. Since $C$ is the closure of the family of heteroclinic orbits connecting $\xi^*$ to $\eta^*$, then any point $\xi\in K$ will reach $U_\epsilon(\eta^*)$ in some finite time $m_\epsilon(\xi)$. Moreover, there exists a neighborhood of $\xi$, denoted by $U_{\delta_\xi}(\xi)$ will contain in $U_\epsilon(\eta^*)$ in time $m_\epsilon(\xi)$, i.e., $$H^{m_\epsilon(\xi)}\left(U_{\delta_\xi}(\xi)\right)\subset U_\epsilon(\eta^*).$$

\noindent Then we can see that $$K\subset \bigcup_{\xi\in K} U_{\delta_\xi}(\xi).$$\noindent Since $K$ is compact, it has a finite open cover, i.e.,
$$K\subset \bigcup_{i=1}^{\bar{m}} U_{\delta_{\xi_i}}(\xi_i).$$
\noindent Choose $\delta=\min\{\epsilon, \min_{1\leq i\leq \bar{m}}\{\delta_{\xi_i}\}\}$. Then any point $\xi\in\tilde{X}$ with $d(\xi, K)<\delta$, then there exists some $m_\xi=m_\epsilon(\xi_i), 1\leq i\leq \bar{m}$, such that $H^{m_\xi}(\xi)\in U_\epsilon(\eta^*)$.

Now assume that the statement of Lemma \ref{repellingC} is not true. Then for any $k$ large enough, there exists some $\xi_k\in\tilde{X}$ and a positive integer $n_k$ such that
\bae\label{false}
d\left(H^{n}(\xi_k), C\right)&<&\frac{1}{k}, \mbox{ for any } n\geq n_k.\eae
\noindent Choose $k>\frac{1}{\delta}$. Then $d\left(H^{n_k}(\xi_k), C\right)< \delta$. We show the contradiction in the following three situations:
\begin{enumerate}
\item If $H^{n_k}(\xi_k)\in U_\epsilon(\eta^*)$, then by Hartman-Grobman-Cushing theorem, $H^{n_k}(\xi_k)$ will exit from $U_\epsilon(\eta^*)$ in some finite time $n_\epsilon(\xi_k,\eta^*)$ and be attracted to a compact neighborhood $B$ in some finite time $l_\epsilon(\xi_k)$. Let $b=d(C, B)$, then we have $$d\left(H^{n_k+n_\epsilon(\xi_k,\eta^*)+l_\epsilon(\xi_k)}(\xi_k), C\right)>b$$ which is a contradiction to \eqref{false}.\\
 \item If $d\left(H^{n_k}(\xi_k),K\right)< \delta$, then there exists some $m_{\xi_k}=m_\epsilon(\xi_i), 1\leq i\leq \bar{m}$, such that $$H^{m_{\xi_k}}\left(H^{n_k}(\xi_k)\right)=H^{n_k+m_{\xi_k}}(\xi_k)\in U_\epsilon(\eta^*),$$which we go back to the first case, therefore, there is a contradiction to \eqref{false}.\\
\item If $H^{n_k}(\xi_k)\in U_\epsilon(\xi^*)$, then by Hartman-Grobman-Cushing theorem, $H^{n_k}(\xi_k)$ will exit from $U_\epsilon(\xi^*)$ in some finite time $n_\epsilon(\xi_k,\xi^*)$, i.e., $$d\left(H^{n_\epsilon\left(\xi_k,\xi^*\right)}\left(H^{n_k}(\xi_k)\right),\xi^*\right)=d\left(H^{n_k+n_\epsilon\left(\xi_k,\xi^*\right)}(\xi_k),\xi^*\right)\geq \epsilon .$$ From \eqref{false}, we have
$$d\left(H^{n_k+n_\epsilon\left(\xi_k,\xi^*\right)}(\xi_k),C\right)<\delta$$which we go back to either the first case or the second case, therefore, there is a contradiction to \eqref{false}.\\
\end{enumerate}
Based on the arguments above, we can conclude that the statement of Lemma \ref{repellingC} is true.
\end{proof}

\begin{lemma}\label{uwr}[Uniform weak persistence]
If the system \eqref{gx0}-\eqref{gy0} satisfies Condition \textbf{C1}-\textbf{C3}. Then there exists some $\delta>0$, such that for any initial condition $\xi_0=(x_0,y_0)\in \tilde{X}$, the system has $$\limsup_{n\rightarrow\infty} x_n>\delta.$$
\end{lemma}
\begin{proof} Since the system satisfies Condition \textbf{C1-C3}, then there exists a compact neighborhood $W\subset S_{y}\cap \tilde{S}$ of the stable periodic-2 orbit $M_y=\{\eta_1,\eta_2\}$ attracting all points $(0,y)\in\tilde{S}$ from Theorem 4.3 (Elaydi and Sacker 2004). Condition \textbf{C3} implies that $M_y$ is transversal unstable, i.e., its external Lyapunov exponent is greater than 1.

Define $P(\xi)=x$ where $\xi=(x,y)\in \tilde{X}$ and
$$r(t,\xi)= \left \{ \begin {array}{cc} \frac{P\left(H^t(\xi)\right)}{P(\xi)}& (\xi\in \tilde{X})\\\noalign{\medskip}  \lim_{\eta\in \tilde{X}\rightarrow \xi\in S_{y}\cap\tilde{S}} \inf \frac{P\left(H^t(\eta)\right)}{P(\eta)} & (\xi\in S_{y}\cap\tilde{S}).\end {array}\right\}$$
Then $r(t,\cdot)$ is lower semicontinuous. For $h>0, t\geq 0$ set \[\alpha(h,t)=\{\xi\in\tilde{X}: r(t,\xi)>1+h\}.\]Then $\alpha(h,t)$ is an open set from the semicontinuity and it has property that $\alpha(h_1,t)\subset \alpha(h_2,t)$ if $ h_1>h_2$. Since $W$ is compact, then there exists a $\bar{h}>0$, and a finite increasing positive integers $k_i, i=1... N$, such that
\[W\subset B\subset \bigcup_{i=1}^{N}\alpha(\bar{h}, k_i).\]where $$B=\overline{U_\epsilon(M_y)}=\{\xi\in \tilde{X}: d(\xi, M_y)\leq \epsilon\}$$ is a compact neighborhood of $W$ in $\tilde{X}$.
We want to show that for any point $\xi=(x,y)\in B \setminus W$, its semi-orbit $\gamma^+(\xi)$ eventually exits from $B$ with $x_n>\epsilon$ for some positive integer $n$. If this is not true, then there exists some point $\xi =(x,y)\in B\setminus W$ such that $H^n(\xi)\in  B \setminus W$ for all $n\geq 0$. Since any point in $B$ belongs to some $\alpha(\bar{h}, k_i), 1\leq i\leq\bar{h}$. This implies that there is a sequence of integer $n_i\rightarrow \infty$ with $n_i-n_{i-1}\in \{k_1,..., k_N\}$ for each $i$ such that
\[P\left(H^{n_i}(\xi)\right)>P\left(H^{n_{i-1}}(\xi)\right)[1+\bar{h}]> x[1+\bar{h}]^{i}\rightarrow \infty \mbox{ since } x>0\mbox{ and } n_i \rightarrow \infty\]which is a contradiction to the fact that all points are attracted to the compact set $D_\epsilon$. Thus, for any point $\xi=(x,y)\in B \setminus W$, its semi-orbit $\gamma^+(\xi)$ eventually exits from $B$ with $x_n>\epsilon$ for some positive integer $n$. Combined with Lemma \ref{repellingC}, we can conclude that for any point $\xi\in \tilde{X}$ that is close enough to $C$, it will enter the compact neighborhood $B$ of $M_y$ and exit from $B$ in some finite time. Therefore, there exists some $\epsilon>0$, such that for any initial condition $\xi\in \tilde{X}$, the system has
\[\limsup_{n\rightarrow\infty} x_n \geq \epsilon .\]
\end{proof}


\end{document}